\documentclass[preprint,1p]{elsarticle}
\usepackage{amssymb}
\usepackage{amsthm}
\usepackage{graphicx}
\usepackage{latexsym}                                                                      
\usepackage{amsmath,amsthm,amsfonts}                                                                                                                   
\usepackage{palatino}                                                                       
\usepackage[mathscr]{eucal}

\newtheorem{thm}{Theorem}[section]

 \newtheorem{lem}[thm]{Lemma}
\newtheorem{cor}[thm]{Corollary}
\theoremstyle{definition}

\newtheorem{example}[thm]{Example}
\theoremstyle{remark}
\newtheorem{remark}[thm]{Remark}
\numberwithin{equation}{section}
\newcommand{\R}{{\Bbb R}}

\newcommand{\N}{{\Bbb N}}

\journal{Mathematical Analysis and Applications}

\begin{document}
\begin{frontmatter}

\title{On existence of semi-wavefronts for a non-local  reaction-diffusion equations with distributed time  delay}
%\tnotetext[t1]{This work was supported by FONDECYT/INICIACION/ N$^o$ 111211457. %\tnoteref{t1}}
\author{ Maitere Aguerrea, Carlos G\'omez}
\ead{maguerrea@ucm.cl, cgomez@ucm.cl}
\address{ Facultad de Ciencias B\'asicas, Universidad Cat\'olica del Maule,Casilla 617,
Talca, Chile.}  

\begin{abstract}
\noindent We  establish  the existence  of  semi-wavefronts solutions for  a non-local delayed reaction-diffusion
equation with   monostable nonlinearity.   The existence result is proved for all speeds $c\geq c_\star$, where the determination  of  $c_\star$ is similar to the calculation of the minimal speed of propagation. The results are applied to some   non-local reaction-diffusion epidemic and population models  with distributed time  delay.

 \end{abstract}

\end{frontmatter}

\section{Introduction.}\label{into}
The main object of study in this paper is  the non-local reaction-diffusion equation
\begin{equation}\label{i1}
u_t(t,x) =  u_{xx}(t,x)- f(u(t,x)) +  \int_0^\infty\int_{\R}K(s,w)g(u(t-s,x-w))dwds,\end{equation}
where  the time $t\geq 0$, $x\in \R$, it is assumed that the non-negative averaging kernel $K$  satisfies the typical condition 
\begin{description}  
\item[\textbf{$H_0$:}]  $K\in L^1(\R_+\times\R)$ and  $\int_0^\infty\int_{\R}K(s,w)dwds=1$. Moreover, for any $c\in \R$, there exists some $\gamma^\#:=\gamma^\#(c)\in (0,+\infty]$ such that  $\int_0^\infty \int_{\R}K(s,w)e^{-z\left(cs+w\right)}dwds<\infty$ for each $z \in [0,\gamma^\#)$ and diverges, if $z>\gamma^\#$.
\end{description}  
Equation (\ref{i1}), with appropriate  $f,g$ and $K$,  is often used to model  ecological and biological processes where the typical interpretation of $u(t,x)$  is the population density of mature species. The last investigations shows that asymmetric kernels can be present in the growth dynamics of single-species population (see, e.g. \cite{Brit,fz,FWZ,gpt,gouk,gouS,gouss,lrw,MSo,OG,swz,tz}). For example, the type of asymmetry can occur in the population marine models when the population juveniles  move by advection as well as diffusion, but the population adults move by diffusion alone (see \cite[Applications]{gpt}).  In this way, we are interested in  equation (\ref{i1}) when  $K$  is asymmetric.

We also assume the following conditions on the monostable  nonlinearity $g$ and the function $f$:
 \begin{description}
 \item[\bf{ $H_1$:}] $g \in C(\R_+,\R_+)$ is  such that $g(0)=0$,  $g(s)>0$ for all $s>0$, and differentiable at 0. 
\item[\bf{ $H_2$:}] $f \in C^1(\R_+,\R_+)$ is strictly increasing with $f(0)=0$, $0<f'(0)<g'(0)$  and  $f(+\infty)>\sup_{s\geq 0} g(s)$.
\end{description}
In this work we will study the existence of semi-wavefront solutions of equation  (\ref{i1}), i.e.  bounded positive continuous non-constant waves  $u(t,x) = \phi(x +ct)$, \,\, propagating with speed $c$,   and  satisfying the boundary condition  $\phi(-\infty)=0$. An important special case of  semi-wavefront is a wavefront, i.e.  positive classical solution $u(t,x) = \phi(x +ct)$ satisfying $\phi(-\infty) = 0$ and $\phi(+\infty) = \kappa$. During the last time, the existence  of  semi-wavefront  or  wavefront solutions for equation  (\ref{i1}) have been  investigated in several papers assuming different conditions on $f$, $K$ and $g$ (see e.g. (see, \cite{map,ft,fhw,gpt,LiWu,ma1,tat,wlr2,wang,xw,ycw}). However, a few works have considered the existence problem for the general non-local reaction-diffusion equation with distributed delay (\ref{i1}) (see \cite{fz,tz,xw}). In any case, for the  non-local reaction-diffusion equation with distributed delay (\ref{i1}), the existence problem of the semi-wavefront is still unsolved in the general case when $K$ is asymmetric. In this paper we apply the techniques of \cite{gpt} to  present a solution of this open problem,  weakening  or  removing some typical restrictions on  nonlinearities. 

Now, we present our   main results:

\begin{thm} \label{2}Assume that {\bf{ $H_0$}}-{\bf{$H_2$}} hold.   If there exists $L\geq g'(0)$ such $g(s)\leq Ls$ for all $s\geq 0$, then there exists $c_\star\in \R$ such that   the equation (\ref{i1}) has a semi-wavefront solution  $u(x,t)=\phi(x+ct)$ propagating  with speed  $c\geq c_\star$. Furthermore,  if equation $f(s)=g(s)$ has only two solutions: $0$ and $\kappa$,  with $\kappa$ being globally attracting with respect to $f^{-1}\circ g$, then  the equation (\ref{i2}) has at least one wavefront $u(x,t)=\phi(x+ct)$ propagating with speed $c\geq c_\star$ such that $\phi(+\infty)=\kappa$. 

\end{thm}

\begin{remark}
We observe  that Theorem \ref{2} shows that  the  condition 
\begin{equation}\label{e11}
g(s)\leq Ls,\,\,\, s\geq 0,
\end{equation}
 when $L=g'(0)$ is not at all obligatory to prove the existence of fast semi-wavefronts solution. Moreover, our result also incorporates asymmetric kernels and the critical case. Thus Theorem \ref{2} improves or completes  the existence results in \cite{fz,tz,xw}, where  the existence was established  for  $g$ satisfying  (\ref{e11})  and  assuming either even or Gaussian kernel $K$. 
\end{remark}
The paper is organized as follows. Section 2 contains some preliminary results. In Section 3, we show some geometric properties of the bounded solutions. In the fourth section, we  our main results are  proved.  In the last section  some applications are presented.
\section{Preliminaries.}
 This section  contains some preliminary results and  transformations are needed to apply the methods of \cite{gpt}.
 
First, note that the profile $y=\phi$ of the  semi-wavefront  solution $u(t,x) = \phi(x +ct)$ to (\ref{i1}) must satisfy the equation
\begin{equation}\label{i2}
 y''(t) - cy'(t)-f(y(t))+ \int_0^\infty\int_{\R}K(s,w)g\left(y(t-cs-w)\right)dwds=0
\end{equation}
for all $t \in \R$. Note that this equation  can be written as 
\begin{equation}\label{ii3}
 y''(t) - cy'(t)-\beta y(t)+f_\beta(y(t))+  \int_0^\infty\int_{\R}K(s,w)g(y(t-cs-
w))dwds=0,
\end{equation}
  where $f_\beta(s)=\beta s- f(s)$ and $\beta$  is chosen large enough such that  $\beta > f'(0)$.  
Thus in order to establish the existence of semi-wavefront solution to (\ref{i1}), we have to prove the existence of positive bounded solution $\phi$  of equation (\ref{i2}), satisfying $\phi(-\infty)=0$.  

Now, being  $\phi$  a positive bounded solution to (\ref{i2}),  it should satisfy the  integral equation
\begin{align} \label{q2}
\phi(t)&=\frac{1}{\sigma(c)}\left(
\int_{-\infty}^te^{\nu(c)(t-s)}(\mathcal G\phi)(s)ds
+\int_t^{+\infty}e^{\mu(c)(t-s)}(\mathcal G\phi)(s)ds\right)\\\nonumber
&= \nonumber\int_\R k_1(t-s)(\mathcal G\phi)(s)ds,\,\, t\in \R,
\end{align}
where $$k_1(s)=(\sigma(c))^{-1}\left\{\begin{array}{cc}e^{\nu(c) s}, & s\geq 0 \\e^{\mu(c) s}, & s<0\end{array}\right.,$$ $\sigma(c)=\sqrt{c^2+4\beta}$,  $\nu(c)<0<\mu(c)$ are the roots of  $z^2-cz-\beta=0$
 and  the operator $\mathcal G$ is defined as $$(\mathcal G\phi)(t):= \int_0^\infty\int_{\R}K(s,w)g(\phi(t-cs-
w))dwds+f_\beta(\phi(t)).$$ 
 Note that $(\mathcal G\phi)(t)$ can be rewritten as 
 \begin{align} \label{q3}\nonumber
 (\mathcal G\phi)(t)&=\int_\R g(\phi(t-r))\int_0^\infty K(s,r-cs)ds dr+f_\beta(\phi(t))\\
&=\int_\R g(\phi(t-r))k_2(r) dr+f_\beta(\phi(t)),
\end{align}
where,  by Fubini's Theorem, $$k_2(r)=\int_0^\infty K(s,r-cs)ds,$$
is well defined for all $r\in \R$.
Finally,   from (\ref{q3}) we get that $\phi$ also must satisfy the equation
\begin{align}\label{dk0} 
\phi(t)&=(k_1*k_2)*g(\phi)(t)+ k_1* f_\beta(\phi)(t), \,t\in \R, \end{align}
where $*$ denotes convolution $(f*g)(t)=\int_\R f(t-s)g(s)ds$. 

 In order to apply some results of \cite{gpt}, we rewrite equation (\ref{dk0}) as
\begin{align}\label{dk1} 
 \phi(t)=\int_Xd\rho(\tau)\int_\R \mathcal N(s,\tau)g(\phi(t-s),\tau) ds, \,t\in \R, \end{align}
where 
\begin{align*}\mathcal N(s,\tau)=\left\{\begin{array}{cc} ( k_1*k_2)(s),& \tau=\tau_1, \\   k_1(s),& \tau=\tau_2, \end{array}\right.\quad  g(s,\tau)=\left\{\begin{array}{cc}  g(s), & \tau=\tau_1, \\ f_\beta(s),& \tau=\tau_2,\end{array}\right. 
\end{align*}
and $X=\{\tau_1,\tau_2\}$. Thus we can invoke the theory developed in \cite{gpt} to prove  the existence of positive bounded solution of (\ref{i2}), vanishing at $-\infty$.

Now we have to introduce several definitions. Let  $c_*$ [respectively, $c_\star$] be the minimal value of  $c$ for which
\begin{equation*}\label{c1}
\chi_0(z,c):=z^2- cz-f'(0)+g'(0)\int_0^\infty \int_{\R}K(s,w)e^{-z\left(cs+w\right)}dwds,
\end{equation*}
\vspace{-0.2cm}
[respectively,
\vspace{-0.1cm}
\begin{equation*}\label{c2}
\chi_L(z,c):=z^2- cz-\inf_{s\geq 0}f'(s)+L\int_0^\infty \int_{\R}K(s,w)e^{-z\left(cs+w\right)}dwds,\, L\geq g'(0)] 
\end{equation*} 
 has  at least one positive root. We observe that  $c_\star\geq c_*$ and the function $\chi_0(z,c)$ is associated with the  linearization of (\ref{i2}) along the trivial equilibrium.  Moreover, we also introduce  the characteristic function $\chi$ associated with the variational equation along the trivial steady state of (\ref{dk1}), by 
$$
\chi(z):= 1-\int_\R \int_X \mathcal N(s,\tau)g'(0,\tau)d\rho(\tau)e^{-sz}ds. 
$$ 
  Observe that 
\begin{align}\label{p}\nonumber
\chi(z)&=1-g'(0)\int_\R \mathcal N(s,\tau_1)e^{-zs}ds-(\beta-f'(0))\int_\R \mathcal N(s,\tau_2)e^{-zs}ds\\\nonumber&= 1-\frac{\beta-f'(0)}{\beta+cz-z^2}-\frac{g'(0)}{\beta+cz-z^2}\int_0^\infty\int_\R K(r,w)e^{-z(rc+w)}dwdr\\&
=-\frac{\chi_0(z,c)}{\beta+cz-z^2}.
\end{align}
 Note that the zeros of function $\chi(z)$ are determined by the roots of characteristic equation $\chi_0(z,c)=0$ and 
\begin{align}\label{pp}
\chi(0)=-\frac{\chi_0(0,c)}{\beta}=\frac{f'(0)-g'(0)}{\beta}<0.
\end{align}
We also will need the following function
\begin{align*}
 \chi_L(z):= 1-\int_\R \int_X \mathcal N(s,\tau)C(\tau)d\rho(\tau)e^{-sz}ds, 
\end{align*}
where 
\begin{align}\label{clip}
C(\tau)=\left\{\begin{array}{cc}  L, & \tau=\tau_1, \\ \beta-\inf_{s\geq 0}f'(s),& \tau=\tau_2,\end{array}\right.
\end{align}
 is measurable function on $(X,\mu)$ with $L\geq g'(0)$.  
 Similarly, note that 
\begin{align}\label{car}
\chi_L(z)= -\frac{\chi_L(z,c)}{\beta+cz-z^2}
\end{align}
and $$\chi_L(0)=\frac{\inf_{s\geq 0}f'(s)-L}{\beta}<0.$$

The properties of the real solutions of the equations $\chi_0(z,c)=\chi_L(z,c)=0$ are shown in the following lemma. We are considering  a  more general equation:  
\begin{equation*}
\mathcal R(z,c):=z^2- cz-q+p\int_0^\infty \int_{\R}K(s,w)e^{-z(cs+w)}dwds=0,
\end{equation*}
where $p>q>0$.  
 \begin{lem} \label{c00}
Suppose that given $c\in \R$, the function $\mathcal R(z,c)$ is defined for all $z$ from some maximal  interval $[0,\delta(c))$,   $\delta(c)\in (0,+\infty]$. Then   there exists $ c^\#\in \R$ such that 
\begin{enumerate}[(i)]
\item for any $c> c^\#$, the function $\mathcal R(z,c)$ has  at least  one positive zero  $z=\lambda_1(c) \in (0,\delta(c))$,  it may have  at most two positive zeros  on $(0,\delta(c))$ and it does not have any negative zero. If the second zero exists, we denote it as $\lambda_2(c)>\lambda_1(c)$. Furthermore, each  $\lambda_j(c)<\mu_q(c)$, where $\mu_q(c)>0$ satisfies the  equation $z^2-cz-q=0$.

\item if $c=c^\#$  and $\lim_{z\uparrow \delta(c^\#)} \mathcal R(z,c^\#)\not=0$, then $\mathcal R(z,c^\#)$  has a unique double root on $(0,\delta(c^\#))$, denoted by $z=\lambda_1(c^\#)$, and $\mathcal R(z,c^\#)>0$ for all $z\not=\lambda_1(c^\#)\in [0,\delta(c^\#))$. 

\end{enumerate}
\end{lem}

\begin{proof} See \cite[Lemma 3.1]{mau}. 
\end{proof}

\begin{example}
We consider a space structured population with maturation effects described by the delays, for example marine species,  where the juveniles  move by advection as well as diffusion, but the adults move by diffusion alone.  If $u(t,x)$ denote the density of the population adult in the location $x\in \R$ and  time  $t$, then $u(t,x)$ can be  represented by the following model:
\begin{align}\label{mar}
u_t(t,x) =  d_au_{xx}(t,x)- \mu_a u(t,x) +  \int_0^\infty\int_{\R}g(u(t-s,x-w))\frac{\mu_je^{\frac{-(w+v_js)^2}{4d_js}-\mu_js}}{2\sqrt{\pi d_j s}}dwds,
\end{align}
where $g$ is the birth function, $d_j,v_j,\mu_j$ are respectively the diffusion rate, the advection velocity and the death rate for juveniles and $d_a, \mu_a$ are respectively the diffusion rate and the death rate for adults population (see \cite{gpt} for more details). Note that spatial asymmetry occurs  in this model  with 
$$K(s,w)=\frac{\mu_je^{\frac{-(w+v_js)^2}{4d_js}-\mu_js}}{2\sqrt{\pi d_j s}}.$$

 By scaling of variables, we can suppose that $d_a=1$. Thus the characteristic function $\chi(z,c)$ associated with the linearization of equation (\ref{mar}) along the trivial steady state is given by the following  form
 \begin{equation*}
\chi(z,c):=z^2- cz-\mu_a+p\int_0^\infty \int_{\R}\frac{\mu_je^{\frac{-(w+v_js)^2}{4d_js}-\mu_js}}{2\sqrt{\pi d_j s}}e^{-z(cs+w)}dwds=0,
\end{equation*}
where $p:=g'(0)>\mu_a$. A simple calculation of the integral allows to obtain that
\begin{equation*}
\chi(z,c):=z^2- cz-\mu_a+\frac{p\mu_j}{\mu_j+(c-v_j)z-d_jz^2}.
\end{equation*}
Note that $ \chi(0,c)=p-\mu_a>0$ and $\lim _{c\downarrow-\infty} \mathcal \chi(z,c)=+\infty$ for $z\in (0,+\infty)$. In addition,
 \begin{equation*}
 \frac{\partial^2 \chi}{\partial z^2}(z,c)>0,\quad z\in [0,+\infty),
 \end{equation*}
 the function  $\chi(z,c)$ is strictly convex with respect to $z$, and hence it has  at most two real zeros  for each $c$. For example, in the particular case when the advection for juveniles is $v_j=0.02$, diffusivity $d_j=100$ and $\mu_j=0.001, \mu_a=0.05$ are the death rates for juveniles and  for adults population, respectively, figures \ref{f1} and \ref{f2} show the behavior of $\chi(z,c)$ when $z\geq 0$ and $p=2$.
\begin{figure}[h]
\begin{minipage}{0.49\linewidth}
\begin{flushleft}
\includegraphics[scale=0.26]{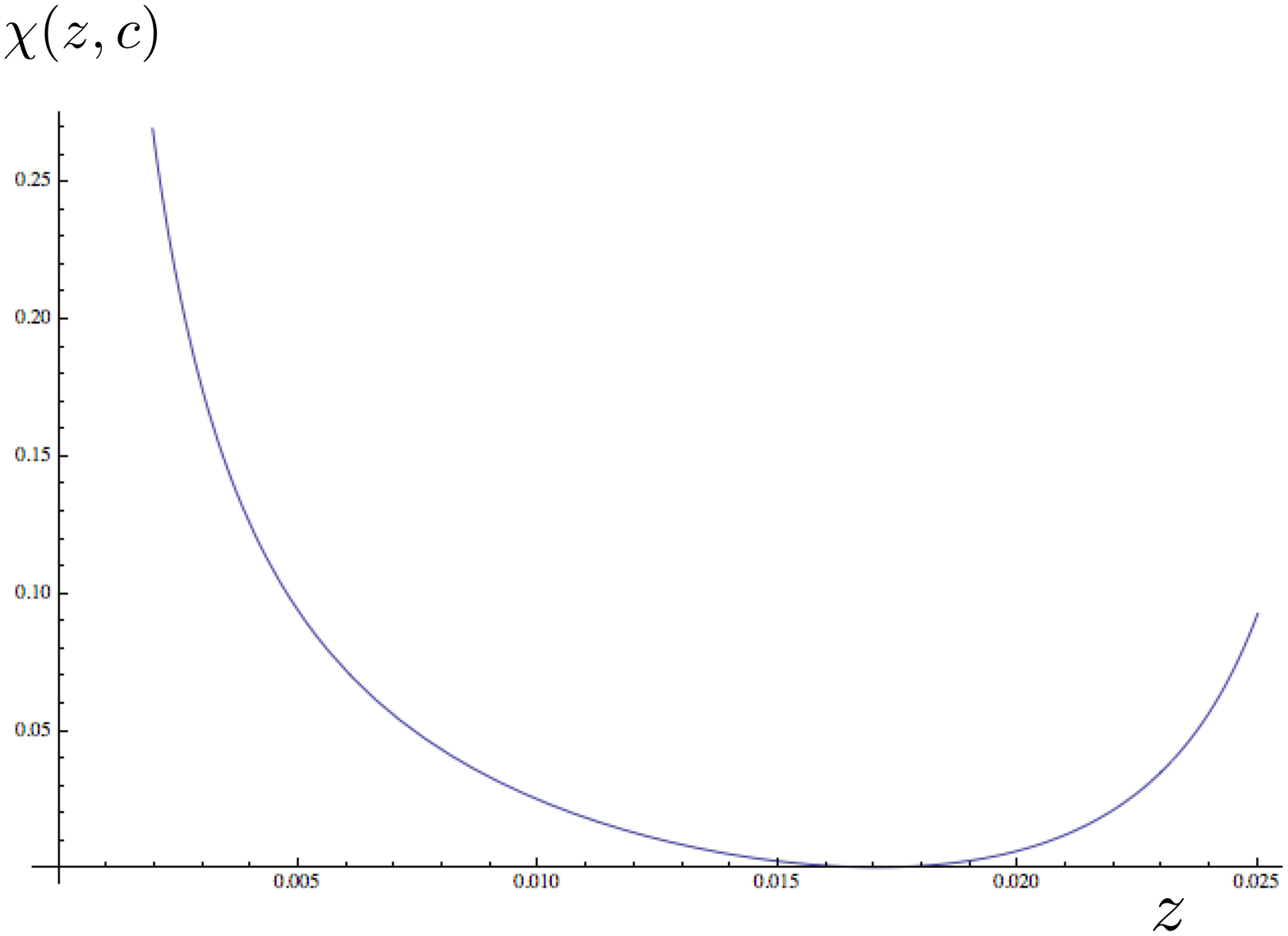}
\caption{{\small  $\chi(z,c), z\geq 0$ for $c=2.854$.}}
\label{f1}
\end{flushleft}
\end{minipage}
\begin{minipage}{0.49\linewidth}
\begin{flushright}
\includegraphics[scale=0.26]{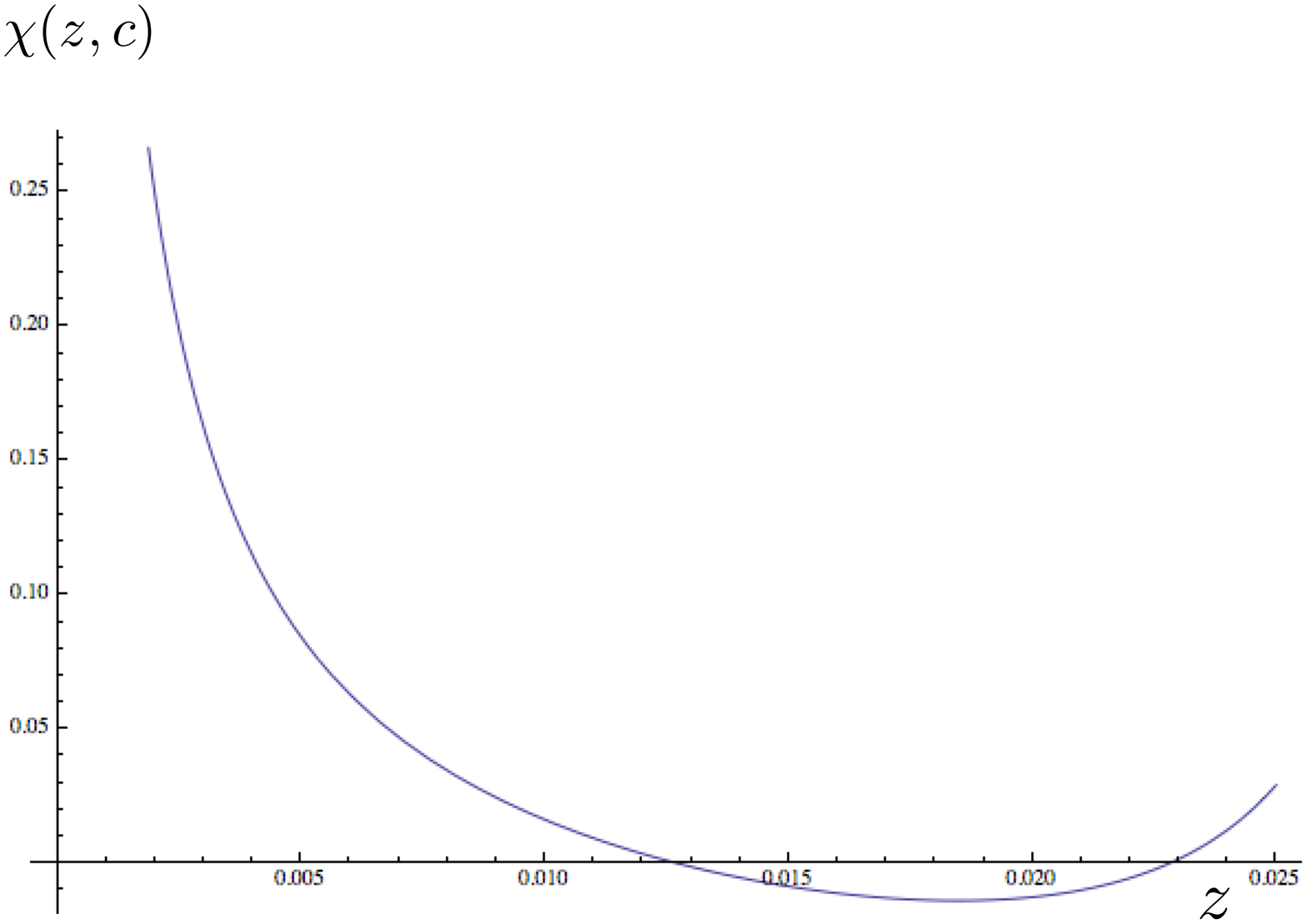}
\end{flushright}
\caption{{\small $\chi(z,c), z\geq 0$ for $c=3$.}}
\label{f2}
\end{minipage}
\end{figure}
\end{example}

\section{Some geometric properties}
Now we show some geometric properties of the bounded solutions to equation (\ref{i2}).
\begin{lem}\label{po}
If $u(t,x)=\phi(x+ct)\geq 0$ is a  bounded solution  of equation (\ref{i2}) such that $\phi$  vanishes at some point, then $\phi\equiv 0$. \end{lem}

\begin{proof}
Suppose that there exists $t_0\in \R$ such that $\phi(t_0)=0$. From (\ref{q2}) we get that 
\begin{align*} 
\phi(t_0)= \int_\R k_1(t_0-s)(\mathcal G\phi)(s)ds=0.
\end{align*}
Since $k_1(t)>0$ and $(\mathcal G\phi)(t)\geq 0$ for all $t\in \R$, we necessarily have $(\mathcal G\phi)(t)= 0$ for all $t\in \R$. 

 In this way,  according to  $f(0)=g(0)=0$ and since $f$ is  continuous  on $[0,\sup_{t\in \R} \phi(t)]$,   we can choose $\beta>0$ sufficiently large   such that $f_\beta(s)=\beta s- f(s)\geq 0$ for all $s\in (0,\sup_{t\in \R} \phi(t)]$. Thus we have
$$\int_\R g(\phi(t-r))k_2(r) dr=f_\beta(\phi(t))=0, \, t\in \R,$$
which implies that $\phi(t)=0$ for all $t\in \R$.
\end{proof}

\begin{lem} \label{to}Assume that {\bf{ $H_0$}}-{\bf{$H_2$}}  and (\ref{e11}) hold. Let $u(t,x)=\phi(x+ct)\geq 0$ be a  bounded solution of equation (\ref{i2}) with speed $c\geq c_*$.  If  $\phi(-\infty)=0$, then  $\lim \inf_{t\to +\infty}\phi(t)\geq\delta(\phi)>0$ for some $\delta(\phi)>0$. \end{lem}
\begin{proof}
First, we observe that  $\chi(0)<0$, by (\ref{pp}). In addition,  the conditions  (\ref{e11}) imply that $g(s,\tau_1)\leq Ls$ for all $s\geq 0$. In this way,  from \cite[Lemma 4.1]{mau}  we get that for $\delta>0$, there exists $\beta=\beta(\delta)>0$ sufficiently large such that   $f_\beta(s)\geq 0$ for all $s\geq 0$ and
$$f_\beta(s) \leq \Big(\beta-\inf_{s\geq 0}f'(s)\Big)s, \quad s\in [0,\delta].$$
Consequently,   $g(s,\tau_2)\leq \Big(\beta-\inf_{s\geq 0}f'(s)\Big)s$ and $g(s,\tau)\leq C_\delta(\tau)s$ for all $s\in [0,\delta]$, where $C_\delta(\tau_1)=L$ and $C_\delta(\tau_2)=\beta-\inf_{s\geq 0}f'(s)$. Note that $C_\delta(\tau)\geq 0$ is a measurable function and
\begin{align*}
\int_X C_\delta(\tau)d\mu(\tau)\int \mathcal N(s,\tau)ds&= L\int_\R(k_1*k_2)(s)ds+(\beta-\inf_{s\geq 0}f'(s))\int_\R k_1(s)ds
\\&=1+\frac{L-\inf_{s\geq 0}f'(s)}{\beta}<+\infty.
\end{align*}
Finally,  from Lemma \ref{po} we obtain  that all the hypotheses of \cite[Theorem 3]{gpt} hold. In consequence,  we can conclude  that   either $\phi(+\infty)=0$ or $\lim \inf_{t\to +\infty}\phi(t)>0$. Since $\phi(-\infty)=0$ and  all real zeros of $\chi(z)=0$ are positive for each $c\geq c_*$ ,  \cite[Corollary 4]{gpt} implies that $\lim \inf_{t\to +\infty}\phi(t)\geq\delta(\phi)>0$ for some $\delta(\phi)>0$.
\end{proof}

\begin{lem}\label{beta2}  Suppose that conditions (\ref{e11})  and {\bf{ $H_0$}}-{\bf{$H_2$}} hold. Furthermore, we also assume that 
\begin{align}\label{con1}
f(s)\geq \inf_{s\geq 0}f'(s)s,\quad s\geq 0.
\end{align}
Let $\phi$ be a positive solution of (\ref{i2}) with speed $c\geq c_\star$ and  $\lambda$ be a positive root of equation $\chi_L(z,c)=0$.  If the solution $\phi$ satisfies the inequality $\phi(t)\leq \delta e^{\lambda t}$ for all $t\in \R$  and  for some $\delta>0$, then  $\phi$ is bounded on $\R$ and 
$$\phi(t)\leq \frac{\sup_{t\geq 0}g(t)}{\inf_{t\geq 0} f'(t)}.$$
\end{lem}
\begin{proof} 
First, it is clear that  $f_\beta(t)\leq (\beta-\inf_{s\geq 0}f'(s))t$ for all $t\geq 0$.
Thus, being $\phi$ a positive solution of (\ref{dk1}),  we obtain that 
\begin{align*}
\phi(t)&\leq 
\sup_{u\geq 0}g(u) \int_\R  (k_1*k_2)(s) ds+\delta(\beta-\inf_{s\geq 0}f'(s))\int_\R  k_1(s)e^{\lambda(t-s)}ds\\ 
&=\frac{\sup_{u\geq 0}g(u)}{\beta}+\delta(\beta-\inf_{s\geq 0}f'(s))e^{\lambda t}\int_\R  k_1(s)e^{-\lambda s}ds.
\end{align*}
In consequence, by definition of function $k_1$, we get that
\begin{align}\label{bu}
\phi(t)&\leq \frac{\sup_{u\geq 0}g(u)}{\beta}+ \frac{\delta(\beta-\inf_{s\geq 0}f'(s))}{\beta+c\lambda-\lambda^2}e^{\lambda t}.
\end{align}
Now, using the inequality (\ref{bu}) and applying this argument again we obtain by induction that 
$$\phi(t)\leq \delta e^{\lambda t}\theta^n+\frac{\rho(1-\gamma^{n+1})}{1-\gamma},\, n\in \N,$$
where $\rho:=\frac{\sup_{u\geq 0}g(u)}{\beta}$, $\theta:=\frac{\beta-\inf_{s\geq 0}f'(s)}{\beta+c\lambda-\lambda^2}$ and $\gamma:=\frac{\beta-\inf_{s\geq 0}f'(s)}{\beta}$. It  is clear that $\gamma<1$. Moreover, since $\lambda$ is  a positive root of equation $\chi_L(z,c)=0$, it is easy  to check that  also $\theta<1$. Thus, by passing to the limit as $n\to \infty$, we obtain the estimate
$$\phi(t)\leq \frac{\rho}{1-\gamma}=\frac{\sup_{u\geq 0}g(u)}{\inf_{s\geq 0} f'(s)}.$$
\end{proof}
We now  establish some properties of $ \mathcal N(s,\tau)$ and $g(s,\tau)$, which will be necessary to apply   methods of \cite{gpt} in order to obtain the uniform persistence of the positive solutions to equation (\ref{i2}).

\begin{lem} \label{Hip} Assume that {\bf{$H_0$}}-{\bf{$H_2$}} hold.  Furthermore, suppose that  the derivative $f'$ is locally bounded. Then the following statements are valid:
\begin{enumerate}[(i)]
\item The function
$\tilde g(v):=\int_\R \int_{X\setminus\{\tau_1\}} g(v,\tau)\mathcal N(s,\tau)d\rho(\tau)ds$ is a monotone increasing function.
 
\item There exists $\xi_2>0$ such that the function $\theta(v):=v-\tilde g(v)$ is strictly increasing on $[0,+\infty)$, and $\theta(\xi_2)>\sup_{v\geq 0}g(v,\tau_1) \int_\R\mathcal N(s,\tau_1)ds$.
\item Define $G(v):=\theta^{-1}(\frac{1}{\beta}g(v,\tau_1))$. Then $G(0)=0$, $0<G(v)<\xi_2$, $v>0$. Furthermore,  $G'(0)$ is finite and $G'(0)>1$.
\end{enumerate}
\end{lem}

\begin{proof}
First, note that 
$$\tilde g(v)=\int_\R g(v,\tau_2)\mathcal N(s,\tau_2)d\rho(\tau)ds=f_\beta(v) \int_\R k_1(s)ds=\frac{f_\beta(v) }{\beta}.$$
Since $f$ is  strictly increasing and $f'$ is locally bounded, the function  $\tilde g(v)=\frac{\beta v-f(v)}{\beta}$ is monotone increasing on $\R$ for some $\beta$  sufficiently large. Moreover, 
$\theta(v)=v-\frac{f_\beta(v)}{\beta}=\frac{f(v)}{\beta}$, $v\geq 0$, is also strictly increasing on $[0,+\infty)$.
Now, consider $\xi_2>0$ such that $f(\xi_2)>\sup_{s\geq 0} g(s)$. Then 
\begin{align*}
\theta(\xi_2)=\frac{f(\xi_2)}{\beta}
>\frac{1}{\beta}\sup_{v\geq 0}g(v)=\sup_{v\geq 0}g(v) \int_\R k_1(s)ds=\sup_{v\geq 0}g(v,\tau_1) \int_\R\mathcal N(s,\tau_1)ds.
\end{align*}
On the other hand,  an easy computation shows that $G(v)=\theta^{-1}(\frac{1}{\beta}g(v,\tau_1))=f^{-1}(g(v))$ and  $G(0)=0$. In addition, if $y=\theta^{-1}(\frac{1}{\beta}g(v_0))>\xi_2$ for some $v_0>0$, then we have
$f(\xi_2)<g(v_0),$ a contradiction. Hence $G(v)<\xi_2$.

Finally, since $g'(0)>f'(0)>0$, it is clear that $1<G'(0)=\frac{g'(0)}{f'(0)}<\infty$,  by Inverse Function Theorem.
\end{proof}
\begin{remark}
We observe that the function $G$ has several other important properties  which are also necessary to prove the uniform persistence of the positive solution to  (\ref{i2}), see \cite[Lemma 5]{gpt}.
\end{remark}
The next lemma establishes the uniform persistence property of semi-wavefront to (\ref{i2}) holds.
\begin{lem}\label{ga} Assume all hypotheses of Lemma \ref{Hip} are fulfilled and suppose the condition (\ref{e11}). Let $u=\phi(x+ct)$ be a positive bounded solution of equation (\ref{i2}) with speed $c\geq c_*$ and $\xi_1\in(0,\xi_2)$, where $\xi_2$ is  as in Lemma \ref{Hip}, such that $\xi_1>\inf_{t\in \R}\phi(t)$. Then  $\phi(-\infty)=0$ and  $\lim \inf_{t\to +\infty} \phi(t)>\xi_1$.
\end{lem}
\begin{proof}
The proof follows from \cite[Theorem 6]{gpt}. In fact, we need the conditions $\textbf{( C )}$, $\textbf{( P )}$ and $\textbf{( N )}$ of \cite{gpt} to be satisfied. In this way, it is clear that the hypothesis $\textbf{( C )}$ holds and $\chi(0)<0$. From Lemma \ref{Hip} we obtain that the hypothesis $\textbf{( N )}$ also holds and the condition $\textbf{( P )}$ is obtained from Lemma \ref{po} . 
\end{proof}

\section{The existence}

Throughout all this section, we assume conditions (\ref{e11}) and (\ref{con1}) hold, and that the speed $c\geq c_\star$.  Let $\lambda>0$ be the leftmost positive root of equation $\chi_L(z,c)=0$ and $m>\lambda$ such that $\chi_L(m,c)<0$, if $c>c_\star$. Note that  for each fix $\tau$ we  have
\begin{align}\label{supl}
g(s,\tau)\leq C(\tau) s, \,s\in \R,
\end{align}
where $C(\tau)$ is defined in (\ref{clip}).

We first consider  $c>c_\star$ and  for some $\delta>0$ we define the functions 
$$\phi^-(t):=\delta e^{\lambda t}(1-e^{(m-\lambda)t})\chi_{\R_-}(t)$$ and $$\phi^+(t):=\delta e^{\lambda t},\,\,t\in \R.$$  
In addition, we will consider the following space
$$X:=\Big\{\phi\in C(\R,\R):||\phi||\\=\sup_{s\leq 0}e^{-\lambda s/2}|\phi(s)|+\sup_{s\geq0}e^{-ms}|\phi(s)|<+\infty\Big\},$$
and the operator $\mathcal A: \mathcal R \to X$, where 
$\mathcal R:= \{\phi\in X:\phi^-(t)\leq \phi(t)\leq \phi^+(t), t\in\R\}$ and
$$\mathcal A \phi (t):=\int_Xd\rho(\tau)\int_\R \mathcal N(s,\tau)g(\phi(t-s),\tau) ds.$$
Note that $\mathcal R\subseteq X$ is closed, bounded and convex set of $X$.
We want  to prove the existence of  $\phi\in \mathcal R$ such that $\mathcal A \phi=\phi$ and $\sup_{s\in \R} \phi(s)<+\infty$. For this, we will prove, in the following Lemma, that  $\mathcal R(\mathcal A)\subseteq \mathcal A$ and $\mathcal A$ is completely continuous map, if we suppose the additional condition
 \begin{description}
 \item[$\textbf{ L:}$] $g(s)=Ls$ and $f_\beta(s)= \Big(\beta-\inf_{s\geq 0}f'(s)\Big)s$  for all $s\in [0,\delta)$.
\end{description}
\begin{lem} \label{er} Assume that {\bf{$H_0$}}-{\bf{$H_2$}}, the condition (\ref{supl}) and $\textbf{ L}$ hold. Then $\mathcal A: \mathcal R \to \mathcal R$ is completely continuous map. Furthermore, $\mathcal A$ has a fix point $\phi$ in $\mathcal R$ such that  $\sup_{t\in \R} \phi(t)<\infty$.
\end{lem}
\begin{proof}
We start with the observation that the proof of this lemma is obtained by similar argument developed in  \cite[Lemma 13]{gpt}. Here we only give the main ideas. In fact, we first define the operator 
$$L\phi(t):=\int_XC(\tau) d\rho(\tau)\int_\R \mathcal N(s,\tau)\phi(t-s) ds.$$

We now will prove that $L\phi^+(t)=\phi^+(t)$ and $\phi^-(t)<L\phi^-(t)$ for all $t\in \R$. For this purpose, since $\chi_L(\lambda)=0$ we see  at once that
\begin{align*}
L\phi^+(t)&=\delta e^{\lambda t}\int_XC(\tau) d\rho(\tau)\int_\R \mathcal N(s,\tau)e^{-\lambda s}ds
\\&=\delta e^{\lambda t}(1-\chi_L(\lambda))=\delta e^{\lambda t}=\phi^+(t),
\end{align*} 
by (\ref{car}). Similarly, since $\chi_L(m)>0$ we also get that
\begin{align*}
L\phi^-(t)&=\delta e^{\lambda t}(1-e^{(m-\lambda)t})+\delta e^{\lambda t}\chi_L(m)\\
&>\delta e^{\lambda t}(1-e^{(m-\lambda)t})=\phi^-(t).
\end{align*} 
An analysis similar to that in the proof of \cite[Lemma 13]{gpt}, with $g'(0,\tau)$ and the root $\lambda$ of $\chi(z)$ replaced by $C(\tau)$ and  root $\lambda$ of $\chi_L(z)$, respectively, shows that $\mathcal A(\mathcal R)\subseteq \mathcal R$ and $\mathcal A$ is completely continuous map. In fact, for $\phi \in\mathcal R$
\begin{align*}
\mathcal A \phi (t)\leq\int_Xd\rho(\tau)\int_\R \mathcal N(s,\tau)C(\tau)\phi(t-s)ds=L\phi(t)\leq L\phi^+(t)=\phi^+(t).
\end{align*}
If there exists  $t_0$ such that $0<\phi^-(t_0)\leq \phi(t_0)$, we would have $t_0\in \R_-$ and $\phi(t_0)\leq \delta$, and hence, $g(\phi (t_0),\tau)=C(\tau)\phi(t_0)\geq C(\tau)\phi^-(t_0)$. In the case that $\phi^-(t_0)=0$, there would be $g(\phi (t_0),\tau)\geq C(\tau)\phi^-(t_0)=0$. Hence
\begin{align*}
\mathcal A \phi (t)\geq \int_Xd\rho(\tau)\int_\R \mathcal N(s,\tau)C(\tau)\phi^-(t-s)ds=L\phi(t)\geq L\phi^-(t)>\phi^-(t).
\end{align*}
This clearly forces
$$\phi^-(t)\leq \mathcal A \phi (t)\leq \phi^+(t), t\in \R.$$
In addition, we observe that  $\mathcal A(\mathcal R)$ is a pre-compact subset of $\mathcal R$, which follows  from
the convergence in $\mathcal R$ is uniform on compact subsets of $\R$. Furthermore, the functions of $\mathcal A(\mathcal R)$ are uniformly bounded on every compact subset of $\R$ and as we have the following estimation
\begin{align*}
|\mathcal A \phi (t+h)-\mathcal A \phi (t)|&\leq \int_X C(\tau) d\rho(\tau)\int_\R \mathcal |N(t+h-u,\tau)-N(t-u,\tau)|\phi(u)du\\
&\leq  \int_X C(\tau) d\rho(\tau)\int_\R \mathcal |N(t+s,\tau)-N(s,\tau)|\phi(t-s)ds\\
&\leq \delta e^{k\lambda} \int_X C(\tau) d\rho(\tau)\int_\R \mathcal |N(s+h,\tau)-N(s,\tau)|e^{-\lambda s}ds, 
\end{align*}
for all $t\in[-k,k]$ and $\phi\in \mathcal R$, they are also equicontinuous, by $|\mathcal A \phi (t+h)-\mathcal A \phi (t)|\to 0, \,\text{as}\,h\to 0$ uniformly on every $[-k,k]$. In this way, the compactness property of $\mathcal A$ and the dominated convergence theorem imply the continuity of $\mathcal A$.
Finally, the Shauder's fixed point theorem implies that $\mathcal A$ has at least one fixed point $\phi\in \mathcal R$. Since $\phi(t)\leq \delta e^{\lambda t}$ and $\phi$ is a positive solution of (\ref{i2}), from Lemma \ref{beta2} we have  $\sup_{t\in \R} \phi(t)<\infty$, which proves the lemma.
\end{proof}

\begin{thm}\label{t1} (Existence of semi-wavefronts) Let  assumptions {\bf{$H_1$}}-{\bf{$H_2$}} hold. Suppose further conditions (\ref{e11}) and (\ref{con1}) hold.  Then the equation (\ref{i2}) has at least one semi-wavefront $u(x,t)=\phi(x+ct)$ propagating with speed $c\geq c_\star$ such that $\phi(-\infty)=0$ and $\lim \inf_{t\to +\infty}\phi(t)>0$.
\end{thm}

\begin{proof}
We define the  continuous functions
\begin{align}
g_n(s)=\left\{\begin{array}{cc}Ls, & s\in[0,1/n], \\\max\left\{\frac{L}{n},g(s)\right\}, & s\geq 1/n,\end{array}\right.
\end{align}
and
\begin{align}
{f_\beta}_{n}(s)=\left\{\begin{array}{cc}\Big(\beta-\inf_{s\geq 0}f'(s)\Big)s, & s\in[0,1/n], \\\max\left\{\frac{\beta-\inf_{s\geq 0}f'(s)}{n},f_\beta(s)\right\}, & s\geq 1/n.\end{array}\right.
\end{align}
Note that $g_n$ and ${f_\beta}_{n}$ satisfy the hypothesis $\textbf{ L}$  whit $\delta=\frac{1}{n}$. Moreover, observe that $g_n$ and ${f_\beta}_{n}$ converge uniformly to $g$ and $f_\beta$ on $\R_+$, respectively. If we now define
\begin{align}
g_n(s,\tau)=\left\{\begin{array}{cc}  g_n(s), & \tau=\tau_1, \\ {f_\beta}_n(s),& \tau=\tau_2,\end{array}\right. 
\end{align}
then $g_n(s,\tau)$ converge uniformly to $g(s,\tau)$ on $\R_+$ for $\tau\in \{\tau_1,\tau_2\}$ and $g_n$ satisfy the condition   $g_n(s,\tau)\leq C(\tau) s,\,s\in \R.$

We now consider the operators $\mathcal A_n: \mathcal R \to X$ defined by 
$$\mathcal A_n \phi (t):=\int_Xd\rho(\tau)\int_\R \mathcal N(s,\tau)g_n(\phi(t-s),\tau) ds.$$
By Lemma \ref{er} we have for each large $n$ the existence of a positive continuos function $\phi_n$ such that $\phi_n(-\infty)=0$ and  $$\mathcal A_n\phi_n(t)=\phi _n(t)=\int_Xd\rho(\tau)\int_\R \mathcal N(s,\tau)g_n(\phi_n(t-s),\tau) ds.$$
Consequently, we have shown that   the functions $\phi_n$ is a positive solution of  equation
\begin{equation}\label{ii5}
 y''(t) - cy'(t)-\beta y(t)+f_{\beta_n}(y(t))+  \int_0^\infty\int_{\R}K(s,w)g_n(y(t-cs-
w))dwds=0,
\end{equation}
whit speed $c>c_\star$.

Proceeding analogously to the proof of  Lemma \ref{beta2}, we can obtain  that $\phi_n$ are bounded functions  such that
$$\phi_n(t)\leq \frac{\max\{L,\sup_{u\geq 0}g(u)\} }{\inf_{t\geq 0} f'(t)}<\infty.$$
Moreover,  for all sufficiently large $n$, and with the same $\xi_1$ and $\xi_2$ given in Lemma \ref{ga} the properties of  Lemma \ref{Hip} hold. In consequence,  we have 
$$\liminf_{t\to +\infty} \phi_n(t)>\xi_1.$$

Consequently, with a similar argument to \cite[Corollary 16]{gpt}, we can prove that $\{\phi_{n}\}$ are equicontinuous on $\R$. Thus there exists a subsequence $\phi_{n_j}$ which converges uniformly on compact sets to some bounded $\phi\in C(\R,\R)$. By Lebesgue's dominated convergence theorem, $\phi$ satisfies the equation \ref{dk1} and $$\phi(-\infty)=0, \quad\liminf_{t\to +\infty}\phi(t)\geq \xi_1>0.$$
Finally, for the case $c=c_\star$, we define $c_n:=\frac{(n+1)c_\star}{n}$. Since $c_n>c_\star$, the pervious result assures the existence of positive bounded solutions $\psi_n$ to (\ref{i2}) such that $\psi_n(-\infty)=0$, $\lim \inf_{t\to +\infty}\psi_n(t)\geq \xi_1$ and 
\begin{align} \label{q+}
\psi_n(t)&=\frac{1}{\sigma(c_n)}\left(
\int_{-\infty}^te^{\nu(c_n)(t-s)}(\mathcal G\psi_n)(s)ds
+\int_t^{+\infty}e^{\mu(c_n)(t-s)}(\mathcal G\psi_n)(s)ds\right)\\\nonumber
&= \nonumber\int_\R k_1(t-s)(\mathcal G\psi_n)(s)ds,\,\, t\in \R,
\end{align}
where $$k_1(s)=(\sigma(c_n))^{-1}\left\{\begin{array}{cc}e^{\nu(c_n) s}, & s\geq 0 \\e^{\mu(c_n) s}, & s<0\end{array}\right.,$$ $\sigma(c_n)=\sqrt{c_n^2+4\beta}$,  $\nu(c_n)<0<\mu(c_n)$ are the roots of  $z^2-c_nz-\beta=0$
 and  the operator $\mathcal G$ is defined as $$(\mathcal G\psi)(t):= \int_0^\infty\int_{\R}K(s,w)g(\psi(t-c_ns-
w))dwds+f_\beta(\psi(t)),$$ 
Now, differentiating (\ref{q+}) we find that 
$$|\psi_n'(t)|\leq \frac{1}{\sigma(c_\star)}\left(\sup_{u\geq 0}g(u)+\frac{\sup_{u\geq 0}g(u)}{\inf_{s\geq 0} f'(s)}\right).$$
In consequence, $\{\psi_n\}$  is pre-compact in the compact open topology of $C(\R,\R)$ and we find a subsequence $\{\psi_{n_j}\}$ which converges uniformly on compacts to some bounded function $\psi\in C(\R,\R).$ In addition, Lebesgue's dominated convergence theorem implies that $\psi$ is a solution of (\ref{i2}) with $c=c_\star$ such that $\psi(-\infty)=0$ and $\lim \inf_{t\to +\infty}\psi(t)\geq \xi_1.$
\end{proof}
\begin{remark}
Note that  if we assume that $L=g'(0)$ and $f'(s)\geq f'(0), t\geq 0$, then  $c_*=c_\star$. In \cite{mau} we have proved that for any $c<c_*$ the equation (\ref{i2}) has no semi-wavefront solution propagating with speed $c$ vanishing at $-\infty$. 
\end{remark}

\begin{cor} (Existence of wavefronts) Assume all conditions of Theorem \ref{t1} are fulfilled.  If equation $f(s)=g(s)$ has only two solutions: $0$ and $\kappa$,  with $\kappa$ being globally attracting with respect to $f^{-1}\circ g$, then  the equation (\ref{i2}) has at least one wavefront $u(x,t)=\phi(x+ct)$ propagating with speed $c\geq c_\star$ such that $\phi(+\infty)=\kappa$.
\end{cor}
\begin{remark}Sufficient conditions to ensure the global stability  of $f^{-1}\circ g$ are given in \cite{tat}.
\end{remark}

 \section {Applications.}
 In this section, we apply Theorem \ref{2}  to some  non-local  reaction-diffusion  epidemic and population models  with distributed time  delay, studied in \cite{AF,FWZ,gouk,OG,tz,wlr2,wl2,wl3,xz}.   

  \vspace{0,3cm}
 {\bf{ An application to  the epidemic dynamics:}} Consider the following reaction-diffusion model with distributed delay
\begin{equation}\label{sy4}
\left\{\begin{array}{c}u_t(t,x) =  du_{xx}(t,x)- f(u(t,x)) + \int_{\R}K(x-y)v(t,y)dy \\\\
v_t(t,x) =  - \alpha v(t,x)+\int_0^\infty g(u(t-s,x))P(ds),  \quad\quad\quad\quad\end{array}\right.
\end{equation}
where $\alpha, d>0$, $x\in \R, t\geq 0$, and $P$ is a probability measure on $\R_+$. The functions $u(t,x)$ and $v(t,x)$ denote the  densities of the infectious agent and the infective human population at a point $x$ in the habitat at time $t$, respectively (see \cite{tz,wl2,wl3,xz}). Note that system (\ref{sy4}) can be seen as a generalization of the systems studied  in the cited works. However, here the nonnegative  kernel $K$ can be asymmetric and normalized by $\int_\R K(w)dw=1$,  and the function $g$ can be non-monotone. By scaling the variables, we can suppose that $d=1$.

Now, suppose that $(u(t,x),v(t,x)) =(\phi(x+ct),\psi(x+ct))$ is a semi-wavefront solution of system (\ref{sy4}) with speed $c$, i.e. the continuous non-constant  uniformly bounded functions $u(t,x)=\phi(x+ct)$  and $v(t,x)=\psi(x+ct)$ are positives  and  satisfy the condition $\phi(-\infty)=\psi(-\infty)=0$.  Then  the wave profiles $\phi$ and $\psi$ must satisfy the following system:
\begin{equation}\label{sy5}
\left\{\begin{array}{c}
\phi''(t)-c\phi'(t)- f(\phi(t)) + \int_{\R}K(u)\psi(t-u)du=0 \\\\
c\psi'(t) +\alpha \psi(t)-\int_0^\infty g(\phi(t-cs))P(ds)=0.  \quad\quad\end{array}\right.
\end{equation}
As it was obtained in \cite{mau},  $\psi$ satisfies 
\begin{align}\label{si1}
\psi(t)
&=\int_0^\infty g(\phi(t-cw))K_2(w) dw,\,\, c\not=0,
\end{align} 
where
$$K_2(w)=\int_0^w e^{-\alpha(w-r)} P(dr),$$
and  if $c=0$, then $\alpha \psi(t)=g(\phi(t))$. In consequence,  $\phi(t)$  should satisfy the  integral equation
\begin{align}  \label{si2}
\phi(t)&=\frac{1}{\sigma(c)}\left(
\int_{-\infty}^te^{\nu(c)(t-s)}(\mathcal G\phi)(s)ds
+\int_t^{+\infty}e^{\mu(c)(t-s)}(\mathcal G\phi)(s)ds\right)\nonumber\\
&= \int_\R k_1(t-s)(\mathcal G\phi)(s)ds,
\end{align}
where $$k_1(s)=(\sigma(c))^{-1}\left\{\begin{array}{cc}e^{\nu(c) s}, & s\geq 0 \\e^{\mu(c) s}, & s<0\end{array}\right.,$$ $\sigma(c)=\sqrt{c^2+4\beta}$,  $\nu(c)<0<\mu(c)$ are the roots of  $z^2-cz-\beta=0$
 and  the operator $\mathcal G$ is defined as $$(\mathcal G\phi)(t):= \int_{\R}K(u)\psi(t-u)du+f_\beta(\phi(t)),\quad f_\beta(s)=\beta s-f(s), \beta>f'(0).$$ 
 Thus,   the profile $\phi$ also must satisfy the equation
\begin{align}\label{si3}
 \phi(t)&=(k_1*k_2)*g(\phi)(t)+ k_1* f_\beta(\phi)(t).
 \end{align}
A similar argument can be applied when $c=0$.

Conversely, if $\phi$ is a semi-wavefront solution of (\ref{si2}) propagating with speed $c$,  then  the function $\psi$ defined by (\ref{si2}) is well defined and $\psi(-\infty)=0$, by the Lebesgue's dominated convergence theorem. Moreover, we get that
$$|\psi(t)|\leq \frac{\sup_{u\geq 0}g(u)}{\alpha}, t\in \R.$$
Since the process developed in \cite{mau} to obtain (\ref{si1}) and (\ref{si2}) is invertible, it follows that $(\phi(t),\psi(t))$ is a semi-wavefront solution to (\ref{sy5}) propagating with speed $c$. In consequence we have the following theorem.
 \begin{lem} The following affirmations are true.
\begin{enumerate}
\item If $(u(t,x),v(t,x)) =(\phi(x+ct),\psi(x+ct))$ is  a semi-wavefront  of (\ref{sy4}) with speed $c$, then  $(\phi(t),\psi(t))$ is a semi-wavefront of  (\ref{sy5}) with speed $c$.
\item Let $\phi(t)$ be  a semi-wavefront  of  (\ref{si3}) with speed $c$.  If $\psi(t)$ is defined by (\ref{si1}), then  $(\phi(t),\psi(t))$ is a semi-wavefront of  (\ref{sy5}) with speed $c$ and $(u(t,x),v(t,x)) =(\phi(x+ct),\psi(x+ct))$ is also a semi-wavefront of (\ref{sy4}).
\end{enumerate}
\end{lem}
Next,  the characteristic function $\chi$ becomes: 
 \begin{align}\label{ppp}\nonumber
\chi(z)&
=-\frac{z^2-cz-f'(0)+\frac{g'(0)}{cz+\alpha}\int_0^\infty e^{-zcr}P(dr)\int_\R K(w)e^{-zw}dw}{\beta+cz-z^2}
\end{align}
when  $cz+\alpha>0$. Consequently,  from \cite{mau}, we obtain
\begin{equation*}\label{c1}
\chi_0(z,c)=z^2-cz-f'(0)+\frac{g'(0)}{cz+\alpha}\int_0^\infty e^{-zcr}P(dr)\int_\R K(w)e^{-zw}dw,
\end{equation*}
and 
\begin{equation*}\label{c2}
\chi_L(z,c)=z^2-cz-\inf_{s\geq 0}f'(s)+\frac{L}{cz+\alpha}\int_0^\infty e^{-zcr}P(dr)\int_\R K(w)e^{-zw}dw. 
\end{equation*}
In this way, let $c_*$ and $c_\star$ be the minimal value of $c$ for which $\chi_0(z,c)=0$ and $\chi_L(z,c)=0$ have at least one positive root,  respectively. Then  we can now formulate the following result:
\begin{thm}\label{apli1}
Let  assumptions {\bf{$H_0$}}-{\bf{$H_2$}} hold.   If $g(s)\leq Ls$ and $f(s)\geq inf_{s\geq 0}f'(s)s,s>0$, then the 
 system (\ref{sy4}) admits at least one semi-wavefront solution $$(u(t,x),v(t,x))=(\phi(x+ct),\psi(x+ct)), \,\, \phi(-\infty)=\psi(-\infty)=0,$$   for each admissible wave speed $c\geq c_\star$. If $inf_{s\geq 0}f'(s)=f'(0)$, then the existence holds for each $c\geq c_*$.
Furthermore, the system (\ref{sy4}) has no  semi-wavefront solution propagating with speed $c<c_*$. \end{thm}

\begin{remark}
Theorem \ref{apli1} completes or  improves  some results  of \cite{tz,wl2,wl3,xz}. In fact, in these references the monotone case was studied, except \cite{wl3}.   It should be noted that in \cite{wl2,xz}, isotropic kernels were  considered.
\end{remark}

 \vspace{0,3cm}
{\bf{An application to the population dynamics:}}  Let $u$ and $v$ denote the numbers of mature and immature population of a single species at time $t\geq 0$, respectively. We will study the system 
\begin{equation}\label{sy3}
\left\{\begin{array}{c}u_t(t,x) =  du_{xx}(t,x)- f(u(t,x)) + \int_0^\infty\int_{\R}K(s,w)g(u(t-s,x-w))dwds\quad\quad\quad\quad \\\\
v_t(t,x) =  Dv_{xx}(t,x)- \gamma v(t,x)+g(u(t,x)) -\int_0^\infty\int_{\R}K(s,w)g(u(t-s,x-w))dwds,\end{array}\right.
\end{equation}
where $\gamma, D,d>0$  and  the nonnegative kernel $K$ can be asymmetric. Note that by scaling the variables, we can suppose that $d=1$. 
 Now, observe that in the system (\ref{sy3}) the first equation can be solved independently of the second. In this way,  if the 
 system (\ref{sy3}) admits   a semi-wavefront solution $$(u(t,x),v(t,x))=(\phi(x+ct),\psi(x+ct)),\,\,\phi(-\infty)=\psi(-\infty)=0,$$  with speed $c$, then $v(t,x)=\psi(x+ct)$ must satisfy the immature equation 
 \begin{align*}
D\psi''(t)-c\psi'(t)- \gamma \psi(t)+(\mathcal H\phi)(t)=0,
\end{align*}
where the operator $\mathcal H$ is defined by
 $$(\mathcal H\phi)(t)=g(\phi(t))-\int_0^\infty\int_{\R}K(s,w)g(\phi(t-cs-w))dwds.$$ 
 If $\phi$ is bounded, we get that $\psi$ can be represented  by 
\begin{align*}
\psi(t)=\int_\R k_1(t-s)(\mathcal H\phi)(s)ds=\int_\R k_1(s)(\mathcal H\phi)(t-s)ds,
\end{align*}
where $$k_1(s)=\left(\sqrt{c^2+4D\gamma}\right)^{-1}\left\{\begin{array}{cc}e^{\tilde\nu(c) s}, & s\geq 0 \\e^{\tilde\mu(c) s}, & s<0\end{array}\right.$$
 and $\tilde\nu(c)<0<\tilde\mu(c)$ are the roots of  $Dz^2-cz-\gamma=0$. In addition, $\mathcal H\in C(\R_+,\R_+)$ and $\mathcal H(0)=0$. Now, if $\phi(-\infty)=0$, then we have $\mathcal H(\phi(-\infty))=0$, and in consequence, the Lebesgue's theorem of dominated convergence implies that $\psi(-\infty)=0$.   Thus we obtain the following lemma.
 
 \begin{lem}
 The second equation of the system (\ref{sy3})  has a semi-wavefront $v(t,x))=\psi(x+ct)$ with $\psi(-\infty)=0$ when $u(t,x)=\phi(x+ct)$ is a semi-wavefront of the first equation of the system (\ref{sy3}).
 \end{lem}

Finally, consider  the characteristic functions $\chi_0(z,c)$ and $\chi_L(z,c)$  associated with the mature equation of system (\ref{sy3})
 and  $c_*,c_\star$  defined in Section \ref{into}. Then the following theorem  is a direct consequence of   Theorem \ref{2}. 

\begin{thm}\label{un}
Let  assumptions   {\bf{$H_1$}}-{\bf{$H_2$}} hold.  If $g(s)\leq Ls$ and $f(s)\geq inf_{s\geq 0}f'(s)s,s>0$, then the 
 system (\ref{sy3}) admits   at least one   semi-wavefront solution $$(u(t,x),v(t,x))=(\phi(x+ct),\psi(x+ct)),\,\,\phi(-\infty)=\psi(-\infty)=0,$$  for each admissible wave speed $c\geq c_\star$. If $inf_{s\geq 0}f'(s)=f'(0)$, then the existence holds for each $c\geq c_*$. Furthermore,  the system (\ref{sy3}) has no  semi-wavefront solution propagating with speed $c<c_*$. 
\end{thm}

\begin{remark}
We note that Theorem \ref{un}  completes or improves some results  of  \cite{FWZ, gouk,tz,wlr2}, where the non-existence or the uniqueness  was established under assumptions  that $K$ is Gaussian or symmetric kernel,  and $g$ monotone.  In \cite{gouk,wlr2} only  the particular cases $f(s)=\beta s^2$ and $g(s)=s$, were studied, and in \cite{tz}, the assumptions were either $f(s)=f'(0)$ or $g(s)=g'(0)s$. 
\end{remark}

\section*{Acknowledgements}  This work was supported by FONDECYT/INICIACION/ Project  11121457.

\end{document}